\documentclass[12pt]{amsart}
\usepackage{amssymb,amsmath,amsthm,amsfonts,mathrsfs,graphicx}
\usepackage{mathtools}
\usepackage[margin=3cm]{geometry}
\usepackage[colorlinks]{hyperref}
\usepackage{tikz}
\mathtoolsset{showonlyrefs=true}

\title[Hydrodynamic Limits]{Hydrodynamic limit of a kinetic flocking model with nonlinear velocity alignment}

\author[McKenzie Black]{McKenzie Black}
\address[McKenzie Black]{\newline Hopkins Extreme Materials Institute, \ 
 Johns Hopkins University, Malone Hall, 3400 North Charles St., Baltimore, Maryland 21218, USA}
\email{Mblack39@jhu.edu}

\author[Changhui Tan]{Changhui Tan}
\address[Changhui Tan]{\newline Department of Mathematics, \ 
 University of South Carolina, 1523 Greene St., Columbia, SC 29208, USA}
\email{tan@math.sc.edu}

\thanks{\textit{Acknowledgment.} This work has been supported by the NSF grants DMS-2108264 and DMS-2238219.}



\newtheorem{theorem}{Theorem}[section]
\newtheorem{lemma}[theorem]{Lemma}

\newtheorem{proposition}[theorem]{Proposition}

\theoremstyle{definition}

\theoremstyle{remark}
\newtheorem{remark}{Remark}[section]

\numberwithin{equation}{section}

\def\R{\mathbb{R}}
\def\T{\mathbb{T}}
\def\pa{\partial}
\def\A{\mathbf{A}}
\def\F{\mathbf{F}}
\def\u{\mathbf{u}}

\def\v{\mathbf{v}}
\def\vbold{\v}
\def\w{\mathbf{w}}
\def\x{\mathbf{x}}
\def\X{\mathbf{X}}
\def\y{\mathbf{y}}
\def\Y{\mathbf{Y}}
\def\z{\mathbf{z}}
\def\grad{\nabla}

\def\eps{\varepsilon}
\def\ue{\mathbf{u}_\eps}
\def\Ee{\mathcal{E}_\eps}
\def\fe{f_\eps}

\def\rhoe{\rho_\eps}
\def\rhot{\tilde{\rho}_\eps}
\def\De{\mathcal{D}_{1\eps}}
\def\Dee{\mathcal{D}_{2\eps}}
\def\Ge{\mathcal{G}_\eps}
\def\Reps{\mathcal{R}_\eps}
\def\Xe{\X_\eps}
\def\S{\mathcal{S}}
\def\etae{\eta_\eps}
\def\I{\mathrm{I}}
\def\II{\mathrm{II}}

\def\V{\mathcal{V}}

\begin{document}
\allowdisplaybreaks

\begin{abstract}
We investigate a class of Vlasov-type kinetic flocking models featuring nonlinear velocity alignment. Our primary objective is to rigorously derive the hydrodynamic limit leading to the compressible Euler system with nonlinear alignment. This study builds upon the work by Figalli and Kang \cite{figalli2018rigorous}, which addressed the scenario of linear velocity alignment using the relative entropy method. The introduction of nonlinearity gives rise to an additional discrepancy in the alignment term during the limiting process. To effectively handle this discrepancy, we employ the monokinetic ansatz in conjunction with the relative entropy approach. Furthermore, our analysis reveals distinct nonlinear alignment behaviors between the kinetic and hydrodynamic systems, particularly evident in the isothermal regime.
\end{abstract}

\maketitle 

\tableofcontents

\section{Introduction}
In this paper, we consider the following Vlasov-type of kinetic flocking model
\begin{equation}\label{eq:kinetic}
 \pa_tf+\v\cdot\grad_\x f+\grad_\v\cdot\big(\F(f) f\big)=0,
\end{equation}
where $f=f(t,\x,\v)$ with $(t,\x,\v)\in\R_+\times\Omega\times\R^d$. The spatial domain $\Omega$ can be either the whole space $\R^d$ or the periodic domain $\T^d$.
The \emph{alignment force} $\F$ is defined as
\begin{equation}\label{eq:kinetF}
    \F(f)(t,\x,\v )=\int_{\Omega\times\R^{d}} \phi(\x-\y)\Phi(\w-\v )f(t,\y,\w)\, d\y\,d\w.
\end{equation}
Here, $\phi$ is the \emph{communication protocol}, representing the strength of pairwise alignment interaction. Throughout the paper, we assume that $\phi$ is radially symmetric, bounded, Lipschitz, and non-increasing along the radial direction. Typical choices are: 
\[\phi(\x)=(1+|\x|)^{-\alpha},\quad \alpha\geq0.\]

The mapping $\Phi : \R^d\to\R^d$ describes the type of alignment. One classical choice is the linear mapping 
\begin{equation}\label{eq:Philinear}
 \Phi(\z)=\z.	
\end{equation}
The corresponding system \eqref{eq:kinetic}-\eqref{eq:Philinear} is often referred as the \emph{Vlasov-alignment system}. It is a kinetic representation of the Cucker-Smale dynamics \cite{cucker2007emergent} that models the flocking phenomena in interacting particle systems. 

A generalization of the Cucker-Smale dynamics was introduced in \cite{ha2010emergent}:
\begin{equation} \label{eq:nCS}
 \begin{cases} 
   \dot\x_i=\v_i,\\
   \dot\v_i= \displaystyle\frac{1}{N}\sum^{N}_{j=1} \phi(\x_i-\x_j)\Phi(\v_j-\v_i),
   \end{cases}
   \quad (\textbf{x}_i,\textbf{v}_i)\in \Omega \times \R^d.
\end{equation}
The system features a nonlinear velocity alignment, where the mapping $\Phi$ takes the form
\begin{equation}\label{eq:Phi}
 \Phi(\z)=|\z|^{p-2}\z,\quad p>2.
\end{equation}
When $p=2$, the mapping $\Phi$ is linear, and \eqref{eq:nCS} reduces to the Cucker-Smale dynamics. For $p\neq2$, the nonlinearity lead to different asymptotic flocking behaviors, as explored in various studies \cite{wen2012flocking,markou2018collision,kim2020complete,black2024asymptotic}.
The system \eqref{eq:kinetic}-\eqref{eq:kinetF} was derived in \cite{carrillo2014local} as a kinetic representation of \eqref{eq:nCS}. The global well-posedness theory was also established in the same work. 

A macroscopic representation of the system \eqref{eq:kinetic}-\eqref{eq:kinetF}  is the following compressible Euler system with alignment interactions:
\begin{equation} \label{eq:EAS}
	\begin{cases}
	\pa_t\rho+\grad_\x\cdot(\rho\u)= 0,\\
    \pa_t(\rho\u)+\grad_\x\cdot(\rho\u\otimes\u)= \rho\A[\rho,\u],
	\end{cases}
\end{equation}
where the alignment force $\A[\rho,\u]$ is defined as
\begin{equation}\label{eq:alignment}
	\A[\rho,\u](t,\x)=\int_{\R^d} \phi(\x-\y)\Phi(\u(t,\y)-\u(t,\x))\rho(t,\y)\,d\y.
\end{equation}
With the linear mapping \eqref{eq:Philinear}, the system \eqref{eq:EAS}-\eqref{eq:alignment} is known as the \emph{Euler-alignment equations}. The system has been extensively investigated in the last decade, see e.g. \cite{tadmor2014critical,carrillo2016critical,shvydkoy2017eulerian,shvydkoy2017eulerian2,do2018global,kiselev2018global,tan2020euler,miao2021global,tan2021eulerian,lear2022existence,leslie2023sticky,li2024global}. For more results on the Euler-alignment system, we refer to the recent book by Shvydkoy \cite{shvydkoy2021dynamics}.

We are interested in the connection between the kinetic equations \eqref{eq:kinetic}-\eqref{eq:kinetF} and macroscopic system \eqref{eq:EAS}-\eqref{eq:alignment}.
The formal derivation was first established in \cite{ha2008particle} when the mapping $\Phi$ is linear \eqref{eq:Philinear}. The Euler-alignment equations were derived by taking zeroth and first moments of $f$ on $\v$, and formally apply the \emph{mono-kinetic ansatz}
\begin{equation}\label{eq:monokinetic}
 f(t,\x,\v)=\rho(t,\x)\delta_{\v=\u(t,\x)},
\end{equation}
where $\delta$ denotes the Dirac delta function.

The rigorous justification of the hydrodynamic limit is discussed by Figalli and Kang in \cite{figalli2018rigorous}. The starting point of their analysis is the kinetic flocking equation:
\begin{equation}\label{eq:kineteps}
\pa_t\fe+\v\cdot\grad_\x \fe+\grad_\v\cdot\big(\F(\fe) \fe\big)=\frac1\eps\grad_\v\cdot\big((\v-\ue)\fe\big).
\end{equation}
In addition to the alignment interaction \eqref{eq:kinetF}, there is another linear relaxation term on the right-hand side of \eqref{eq:kineteps}. As the parameter $\eps$ tends to zero, the relaxation term enforces the mono-kinetic ansatz \eqref{eq:monokinetic}. This relaxation term was introduced in \cite{karper2014strong,kang2015asymptotic}, viewed as local alignment.

 The macroscopic density and momentum associated with $\fe$ are denoted by $\rhoe$ and $\rhoe\ue$ respectively. These are defined as the zeroth and first moments of $\fe$ with respect to velocity $\v$, expressed as:
\begin{equation}\label{eq:rhoepe}
    \rhoe(t,\x)=\int_{\R^d} \fe(t,\x,\v)\, d\v , \quad  \rhoe\ue(t,\x)=\int_{\R^d} \vbold \fe(t,\x,\v)\, d\v.
\end{equation}
Hence, the velocity $\v$ is relaxed to the macroscopic velocity $\ue$ given by:
\begin{equation}\label{eq:ue}
    \ue(t,\x):=\frac{\int_{\R^{d}} \v \fe(t,\x,\v )\,d\v}{\int_{\R^{d}} \fe(t,\x,\v)\,d\v}.
\end{equation}

From \eqref{eq:kineteps} with \eqref{eq:kinetF} and \eqref{eq:Philinear}, the dynamics of $\rhoe$ and $\rhoe\ue$ can be derived, resulting in the following system:
\begin{equation}\label{eq:rhouelinear}
 \begin{cases}
  \pa_t\rhoe+\grad_\x\cdot(\rhoe \ue)= 0,\\
  \pa_t(\rhoe\ue)+\grad_\x\cdot(\rhoe\ue\otimes\ue+\Reps)=\rhoe(\x)\displaystyle\int_{\Omega} \phi(\x-\y)(\ue(\y)-\ue(\x))\rhoe(\y)d\y,
   \end{cases}
\end{equation}
where $\Reps$ represents the Reynold's stress tensor 
\begin{equation}\label{eq:Reps}
   \Reps(t,\x)=\int_{\R^d}(\v -\ue)\otimes(\v -\ue)\fe(t,\x,\v )\,d\v. 	
\end{equation}

Formally applying the mono-kinetic ansatz \eqref{eq:monokinetic} to \eqref{eq:Reps} results in $\Reps\equiv0$. Consequently, \eqref{eq:rhouelinear} transforms into the pressure-less Euler-alignment equations \eqref{eq:EAS} with \eqref{eq:Philinear}.

The rigorous derivation of the hydrodynamic limit, however, is non-trivial. In \cite{figalli2018rigorous}, a \emph{relative entropy method} is employed to rigorously establish the limit:
\[\fe(t,\x,\v)\to f(t,\x,\v)=\rho(t,\x)\delta_{\v=\u(t,\x)}\]
in an appropriate sense. Here, $(\rho,\rho\u)$ constitutes the solution to the Euler-alignment equations.

In this paper, our primary objective is to generalize the findings on the hydrodynamic limit to the case of nonlinear velocity alignment described by \eqref{eq:Phi} with $p>2$.

The formal derivation for the hydrodynamic limit involving general choices of $p$ has recently been undertaken by Tadmor in \cite{tadmor2023swarming}.
The alignment force \eqref{eq:alignment} with $\Phi$ in \eqref{eq:Phi} is  referred as \emph{$p$-alignment}.
The limiting system \eqref{eq:EAS} has been less thoroughly understood compared to the Euler-alignment equations (when $p=2$), primarily due to the introduced nonlinearity. 
Recent investigations, as reported in \cite{tadmor2023swarming,lu2023hydrodynamic,black2024asymptotic}, have shed light on intriguing asymptotic behaviors stemming from the nonlinear nature of $p$-alignment.

One significant challenge in rigorously justifying this limit arises from the nonlinearity, which introduces an additional term in the momentum equation:
\begin{equation}\label{eq:rhoue}
  \pa_t(\rhoe\ue)+\grad_\x\cdot(\rhoe\ue\otimes\ue+\Reps)=\rhoe(\x)\displaystyle\int_{\Omega} \phi(\x-\y)\Phi(\ue(\y)-\ue(\x))\rhoe(\y)d\y+\Ge.
\end{equation}
The discrepancy term $\Ge=\Ge(t,\x)$ takes the form
\begin{equation}\label{eq:Ge}
\int_{\Omega\times\R^{2d}}\phi(\x-\y)\big(|\w-\v|^{p-2}-|\ue(\y)-\ue(\x)|^{p-2}\big)(\w-\v)\fe(\y,\w)\fe(\x,\v)\,d\y d\v d\w.
\end{equation}
We refer to Section \ref{sec:formal} for a formal derivation of the discrepancy $\Ge$. It is noteworthy that when the mapping $\Phi$ is linear ($p=2$), the discrepancy $\Ge$ does not exist, i.e., $\Ge\equiv0$. Conversely, when $p>2$, obtaining additional control over the term $\Ge$ becomes imperative.

Formally inserting the mono-kinetic ansatz \eqref{eq:monokinetic} into \eqref{eq:Ge} results in $\Ge\equiv0$. Consequently, the expectation is that  the discrepancy vanishes as $\eps\to0$. However, achieving a rigorous limit requires delicate control of $\Ge$ through the relative entropy and the linear relaxation. This undertaking will be thoroughly investigated in the course of this paper.

We are ready to state our main result on the rigorous derivation of the hydrodynamic limit.
\begin{theorem}\label{thm:hydro}
Let $\fe$ and $(\rho,\u)$ be the solutions to \eqref{eq:kineteps} and \eqref{eq:EAS} respectively in the time interval $[0,T_*]$, with well-prepared initial data. Then
\[
\fe(t,\x,\v)\rightharpoonup f(t,\x,\v)=\rho(t,\x)\delta_{\v=\u(t,\x)},\quad	\text{as}\,\, \eps\to0.
\]
\end{theorem}
\begin{remark}
 The complete details of the theorem, including the definitions of solutions to the systems, the interpretation of well-prepared initial data, and the notation of convergence, will be presented later in the main context. Refer to Theorem \ref{thm:main} for the comprehensive results.
\end{remark}
\begin{remark}
 In this paper, we primarily focus on the $p$-alignment nonlinearity given by \eqref{eq:Phi}. However, it is worth noting that Theorem \ref{thm:hydro} remains valid for a more general class of nonlinearities of the form
 \begin{equation}\label{eq:Phigeneral}
  \Phi(\z) = h(|\z|)\z,	 	
 \end{equation}
 where $h$ is an increasing function on $\R_+$ such that $h(0)=0$, and $h\in C^q(\R_+)$ for $q\in(0,1]$. 
 For details on how H\"older/Lipschitz continuity on $b$ is incorporated into our argument, refer to Remark \ref{rmk:Phi}.
\end{remark}

It is crucial to underscore that the mono-kinetic ansatz \eqref{eq:monokinetic} plays a pivotal role in establishing the same nonlinearity $\Phi$ in the resulting system \eqref{eq:EAS}.
Notably, the alignment interaction within the limiting system does not necessarily conform to the same nonlinearity in general. We illustrate this aspect in the subsequent discussion.

One commonly considered equilibrium state is the Gaussian function
\begin{equation}\label{eq:isothermal}
 f(t,\x,\v)=\rho(t,\x)\cdot (2\pi)^{-\frac{d}2} e^{-\frac{|\v-\u(t,\x)|^2}{2}},
\end{equation}
known as the \emph{isothermal ansatz}. Plugging in this ansatz to \eqref{eq:Reps} would yield
\[\Reps(t,\x)=\rhoe(t,\x)\,\mathbb{I}_d,\]
where $\mathbb{I}_d$ denotes the $d$-by-$d$ identity matrix.
In the case of a linear mapping $\Phi$, the limiting system corresponds to the Euler-alignment equations with isothermal pressure. Specifically, the momentum equation takes the form
\[
  \pa_t(\rho\u)+\grad_\x\cdot(\rho\u\otimes\u)+\grad_\x\rho=\rhoe(\x)\displaystyle\int_{\Omega} \phi(\x-\y)(\ue(\y)-\ue(\x))\rhoe(\y)d\y.
\]
The rigorous derivation of this type of hydrodynamic limit has been explored in \cite{karper2015hydrodynamic}, stemming from the following Vlasov-Fokker-Planck equation with alignment
\begin{equation}\label{eq:feFP}
\pa_t\fe+\v\cdot\grad_\x \fe+\grad_\v\cdot\big(\F(\fe) \fe\big)=\frac1\eps\grad_\v\cdot\big((\v-\ue)\fe\big)+\frac1\eps\Delta_\v\fe,	
\end{equation}
where the right-hand side enforces the isothermal ansatz \eqref{eq:isothermal} as $\eps\to0$.

For a nonlinear mapping $\Phi$ in \eqref{eq:Phi} with $p>2$, a crucial observation is that $\Ge$ does not tend to zero as $\eps\to0$. Consequently, the alignment interaction in the limiting system is \emph{not} $p$-alignment. We present the isothermal hydrodynamic limit in Section \ref{sec:isothermal}, leaving the rigorous justification for future investigation.

We would also like to highlight a distinct type of communication protocol $\phi$, known as \emph{singular communications}, where $\phi$ is unbounded at the origin. For instance, $\phi(\x)=|\x|^{-\alpha}$ with $\alpha>0$. The kinetic equation \eqref{eq:kinetic}-\eqref{eq:kinetF} with singular communication $\phi$ and linear mapping $\Phi$ has been investigated in \cite{mucha2018cucker,choi2023local}. A recent paper \cite{fabisiak2023inevitable} suggests that singular communications enforce mono-kinetic ansatz \eqref{eq:monokinetic}. A rigorous study of the hydrodynamic limit in this context would be interesting. Some relevant studies have been conducted by Poyato and Soler in \cite{poyato2017euler}.

The rest of the paper is organized as follows. 
Section \ref{sec:VA} presents some preliminary results on the kinetic flocking equation \eqref{eq:kineteps}.
Section \ref{sec:hydro} consists of a formal derivation of the hydrodynamic limit from \eqref{eq:kineteps} to \eqref{eq:EAS}, a local well-posedness theory for the limiting system \eqref{eq:EAS}, and the complete statement of our main result, Theorem \ref{thm:hydro}.
The proof of the theorem is furnished in Section \ref{sec:Main}, leveraging the relative entropy method. The key innovation lies in controlling the discrepancy $\Ge$ through the mono-kinetic structure enforced by the linear relaxation.
Finally, Section \ref{sec:isothermal} discusses the hydrodynamic limit with the isothermal ansatz \eqref{eq:isothermal}. Notably, the limiting system has an alignment force that is different from the $p$-alignment.
 
\section{The Vlasov-alignment system}\label{sec:VA}
In this section, we state a collection of preliminary results on the Vlasov-alignment system \eqref{eq:kineteps}. Recall the dynamics
\begin{equation}\label{eq:VAeps}
\begin{cases}
\pa_t\fe+\v\cdot\grad_\x \fe+\grad_\v\cdot\big(\F(\fe) \fe\big)=\displaystyle\frac1\eps\grad_\v\cdot\big((\v-\ue)\fe\big),\\
\fe(0,\x,\v)=\fe^0(\x,\v),	
\end{cases}
\end{equation}
where the alignment force $\F$ is defined in \eqref{eq:kinetF}.

We assume non-negative and compactly supported initial data
\begin{equation}\label{eq:feinit}
	\fe^0(\x,\v)\geq0,\,\, \text{diam}(\text{supp}_{\x}\fe^0)\leq \S^0<\infty,\,\text{and}\,\,\text{diam}(\text{supp}_{\v}\fe^0)\leq \V^0<\infty,
\end{equation}
where $(\S^0, \V^0)$ are finite numbers that are independent with $\eps$. 
For simplicity, we assume unit total mass
\begin{equation}\label{eq:fenormal}
 \int_{\Omega\times\R^d}\fe^0(\x,\v)\,d\x d\v=1.
\end{equation}
Note that the total mass is preserved in time.

\subsection{Local and global well-posedness}
The global well-posedness theory for classical solutions to \eqref{eq:kinetic} follows from standard argument for Vlasov-type equations. It requires Lipschitz continuity in $(\x,\v)$ of the forcing terms $\F(\fe)$. See \cite[Theorem 3.3]{ha2008particle} for the case $p=2$, and \cite{carrillo2014local} for more general discussions.

For equation \eqref{eq:VAeps} with the linear relaxation term, additional a priori control of $\grad_x\ue$ is required to ensure Lipschitz continuity of the term $\frac1\eps(\v-\ue)$.
\begin{proposition}\label{prop:WPc}
 Let $\fe^0\in(C^1\cap W^{1,\infty})(\Omega\times\R^d)$ and satisfies \eqref{eq:feinit}. There exists a unique classical solution $\fe\in C^1([0,T)\times\Omega\times\R^d)$ to equation \eqref{eq:VAeps}, provided
 \begin{equation}\label{eq:uecond}
 \|\grad_\x\ue\|_{L^\infty([0,T)\times\Omega)}<+\infty.	
 \end{equation}
 \end{proposition}
 
In \cite{karper2013existence}, the authors construct weak solutions to \eqref{eq:VAeps} with $p=2$ by regularizing $\ue$ and obtaining uniform control analogous to \eqref{eq:uecond}. We state the following version of their theorem for general $p$-alignment.

\begin{proposition}\label{prop:WPw}
 Let $\fe^0\in L^\infty(\Omega\times\R^d)$ and satisfies \eqref{eq:feinit}. Then there exists a weak solution $\fe\in L^\infty([0,T)\times\Omega\times\R^d)$ to equation \eqref{eq:VAeps} in the sense of distribution, that is, 
\begin{align}\label{eq:weaksol}
&\int_0^T\int_{\Omega\times\R^d}\fe\Big(\pa_t\varphi+\v\cdot\grad_\x\varphi+\F(\fe)\cdot\grad_\v\varphi+\frac1\eps(\ue-\v)\cdot\grad_\v\varphi\Big)d\x d\v dt\\
&\qquad\qquad+\int_{\Omega\times\R^d}\fe^0\varphi(0,\cdot)\,d\x d\v=0,\nonumber
\end{align}
 for any $\varphi\in C_c^\infty([0,T)\times\Omega\times\R^d)$.
\end{proposition}

\begin{remark}
Shvydkoy \cite{shvydkoy2021dynamics} studied the hydrodynamic limits from \eqref{eq:kinetic} with a regularized local relaxation term $\frac1\eps\grad_\v\cdot((v-\u_\eps^\delta)\fe)$, so that condition \eqref{eq:uecond} holds for any fixed $\delta>0$. Applying Proposition \ref{prop:WPc}, the kinetic equation has a unique classical solution $\fe^\delta$. The hydrodynamic limit can then be studied by letting $\eps, \delta\to0$ appropriately.
For \eqref{eq:VAeps}, there is no uniqueness guaranteed for the weak solution. We will show the hydrodynamic limit starting from any weak solutions $\fe$ that satisfy \eqref{eq:weaksol}.
\end{remark}

We define the kinetic energy (or entropy)
\begin{equation}\label{eq:Eeps}
  \Ee(t)=\frac12\int_{\Omega\times\R^d}|\v|^2\fe(t,\x,\v)\,d\x d\v.
\end{equation}
The energy is dissipated by the alignment force $\F$, as well as the local relaxation. Define the kinetic enstrophy
\begin{align}
  \De(t)&=\frac{1}{2}\int_{\Omega^2\times\R^{2d}} \phi(\x-\y)|\w-\v|^p\fe(t,\x,\v)\fe(t,\y,\w)\,d\x d\y d\v d\w,\label{eq:D1}\\
  \Dee(t)&=\int_{\Omega\times\R^d}|\v -\ue|^2\fe(t,\x,\v)\,d\x d\v.\label{eq:D2}
\end{align} 
We have the following bound on the energy dissipation.

\begin{proposition}\label{prop:energy}
For any $\eps>0$, let $\fe$ be a weak solution to \eqref{eq:VAeps}. We have
    \begin{equation}\label{eq:energydecay}
        \frac{d}{dt}\Ee(t) \leq -\De(t)- \frac{1}{\eps}\Dee(t),
    \end{equation}
where the energy $\Ee$ and enstropy $\De$, $\Dee$ are defined in \eqref{eq:Eeps}, \eqref{eq:D1} and \eqref{eq:D2}, respectively.
\end{proposition}
\begin{proof}
Suppose $\fe$ is a classical solution to \eqref{eq:VAeps}. We utilize \eqref{eq:VAeps} and get 
    \begin{align*}
             \frac{d}{dt}\Ee(t)=&\,\frac{1}{2}\int_{\Omega\times\R^d}|\v|^2\partial_t\fe \,d\x d\v \\
             =&\,-\int_{\Omega\times\R^d} \frac{|\v|^2\v}{2} \cdot \grad_\x \fe\,d\x d\v -\int_{\Omega\times\R^d} \frac{|\v|^2}{2}\grad_\v\cdot (\F(\fe)\fe)\,d\x d\v\\
             &\,+\frac{1}{\eps}\int_{\Omega\times\R^d} \frac{|\v|^2}{2}\grad_\v \cdot((\v -\ue)\fe)\,d\x d\v\\
             =&\,\int_{\Omega\times\R^d} \v\cdot\F(\fe)\fe\,d\x d\v -\frac{1}{\eps}\int_{\Omega\times\R^d} \v\cdot (\v -\ue)\fe\,d\x d\v\\
             =&\,\int_{\Omega^2\times\R^{2d}} \phi(\x-\y)\,\v\cdot(\w-\v)|\w-\vbold |^{p-2}\fe(\x,\v )\fe(\y,\w) \, d\x d\y d\v d\w\\
             &\,-\frac{1}{\eps}\int_{\Omega\times\R^d} (\v-\ue)\cdot(\v-\ue)\fe\,d\x d\v\\
             =&\,-\De-\frac{1}{\eps}\Dee.
    \end{align*}
Here, we have used the identity $\int_{\R^d}(\v-\ue)\fe\,d\v=0$ in the penultimate equality, and symmetrized in $(\v,\w)$ for the last equality.     
        
For weak solutions, we apply the calculation above to a sequence of smooth approximations, and pass to the limit to obtain the inequality \eqref{eq:energydecay}.
\end{proof}

\subsection{Asymptotic flocking behavior}
In this part, we present several properties of the solution $\fe$ to \eqref{eq:VAeps} concerning its support in $(\x,\v)$. We define the variation of position and velocity as follows:
\begin{equation}\label{eq:DV}
\S_\eps(t)= \text{diam}(\text{supp}_\x\fe(t)),\quad
\V_\eps(t)= \text{diam}(\text{supp}_\v\fe(t)). 
\end{equation}

We begin by stating a maximum principle that will be utilized throughout this paper.
\begin{proposition}\label{prop:MP}
 Suppose $\fe$ is a weak solution to \eqref{eq:VAeps}, with initial data $\fe^0$ satisfying \eqref{eq:feinit}. Then we have
 \begin{equation}\label{eq:MP}
 	\V_\eps(t)\leq \V^0,\quad \text{and}\quad \S_\eps(t)\leq \S^0+t\V^0,
 \end{equation}
 for any $\eps>0$ and $t\in[0,T)$.
\end{proposition}
The maximum principle \eqref{eq:MP} holds for general nonlinearity \eqref{eq:Phigeneral}. For $p$-alignment \eqref{eq:Phi}, refined estimates can be obtained.
Indeed, a similar argument as in \cite{black2024asymptotic} yields the following system of inequalities on $(\S_\eps,\V_\eps)$:
\begin{equation}\label{eq:SDDI}
\begin{cases}
	\S_\eps'(t)\leq \V_\eps(t),\\
	\V_\eps'(t)\leq -2^{2-p}\phi(\S_\eps(t))\V_\eps(t)^{p-1},
\end{cases}
\quad\text{with}\quad
\begin{cases}
	\S_\eps(0)\leq \S^0,\\
	\V_\eps(0)\leq \V^0.
\end{cases}
\end{equation}
The analysis of \eqref{eq:SDDI} reveals asymptotic alignment and flocking behavior in the system. For instance, assuming $\phi$ has a positive lower bound $\underline{\phi}>0$, we obtain
\[\V_\eps'(t)\leq -2^{2-p}\underline{\phi}\V_\eps(t)^{p-1},\quad \V_\eps(0)\leq\V^0,\]
which implies velocity alignment with an algebraic decay rate for $p>2$. Specifically, we have
\[\V_\eps(t)\leq \Big((\V^0)^{-(p-2)}+2^{2-p}(p-2)\underline{\phi}\,t\Big)^{-\frac{1}{p-2}}\lesssim t^{-\frac{1}{p-2}}.\]
This is notably different from the case of linear mapping ($p=2$), where the decay rate is exponential.

A more interesting setup occurs when $\Omega=\R^d$ and $\phi$ decays to zero like $\phi(r)\sim r^{-\alpha}$. The system \eqref{eq:VAeps} exhibits different asymptotic behaviors for various choices of $p$ and $\alpha$. Detailed discussions are provided in \cite{black2024asymptotic}.

\section{Hydrodynamic limit}\label{sec:hydro}

\subsection{A formal derivation}\label{sec:formal}
We start with a formal derivation of the hydrodynamic limit from the kinetic system \eqref{eq:kineteps} to the Euler-alignment system \eqref{eq:EAS}. The derivation was first established in \cite{ha2008particle} for the linear alignment case $p=2$, and in \cite{tadmor2023swarming} for general nonlinear alignment with $p>2$. For the sake of completeness, we present a formal derivation in this paper, under our notations.

We start with computing the zeroth and first moments of $\fe$. Integrating \eqref{eq:kineteps} in $\v$ yields the continuity equation
\[\pa_t\rhoe+\grad_\x\cdot(\rhoe \ue) = 0.\]
Multiplying \eqref{eq:kineteps} by $\v$ and integrating in $\v$, we obtain the momentum equation
\begin{equation}\label{eq:momentumeps}
\pa_t(\rhoe\ue)+\grad_x\cdot\int_{\R^{d}}\v \otimes \v \fe\,d\v =\int_{\R^d}\F(\fe)\fe \,d\v.	
\end{equation}
We rewrite the second moment by
\[\int_{\R^{d}}\v \otimes \v \fe\,d\v=\rhoe\ue\otimes\ue+\mathcal{R}_\eps,\]
where $\Reps$ is the Reynold's stress tensor defined as
\[\Reps=\int_{\R^d}(\v -\ue)\otimes(\v -\ue)\fe\,d\v.\]

For the alignment term on the right hand side of \eqref{eq:momentumeps}, if $p=2$, it can be represented by the macroscopic quantities $(\rhoe,\ue)$. Indeed, in this case, we have
\begin{align*}
\int_{\R^d}\F(\fe)\fe \,d\v =&\int_{\Omega\times\R^d\times\R^d}\phi(\x-\y)(\w-\v)\fe(\x,\v)\fe(\y,\w)\,d\y d\v d\w\\
=&\int_\Omega\phi(\x-\y)(\ue(\y)-\ue(\x))\rhoe(\x)\rhoe(\y)\,d\y=\rhoe\A(\rhoe,\ue).
\end{align*}
When $p>2$, the alignment term depends on higher moments of $\fe$. We decompose $\F$ into two parts
\begin{align*}
    &\F(\fe)(\x,\v)=\int_{\Omega\times\R^{d}} \phi(\x-\y)|\ue(\y)-\ue(\x)|^{p-2}(\w-\v)\fe(t,\y,\w)\, d\y d\w\\
    &\quad +\int_{\Omega\times\R^{d}} \phi(\x-\y)\big(|\w-\v|^{p-2}-|\ue(\y)-\ue(\x)|^{p-2}\big)(\w-\v)\fe(t,\y,\w)\, d\y d\w\\
    &=:\F_1(\fe)(\x,\v)+\F_2(\fe)(\x,\v).
\end{align*}
The first term $\F_1$ is linear in $\v$. Hence, we have
\begin{align*}
\int_{\R^d}\F_1(\fe)\fe \,d\v =&\int_\Omega\phi(\x-\y)|\ue(\y)-\ue(\x)|^{p-2}(\ue(\y)-\ue(\x))\rhoe(\x)\rhoe(\y)\,d\y\\
=&\rhoe(\x)\int_\Omega\phi(\x-\y)\Phi(\ue(\y)-\ue(\x))\rhoe(\y)\,d\y=\rhoe\A(\rhoe,\ue).
\end{align*}
For the remaining term $\F_2$, we denote
\begin{align*}
&\Ge(t,\x):=\int_{\R^d}\F_2(\fe(t,\x,\v))\fe(t,\x,\v) \,d\v\\
&=\int_{\Omega\times\R^{2d}}\phi(\x-\y)\big(|\w-\v|^{p-2}-|\ue(\y)-\ue(\x)|^{p-2}\big)(\w-\v)\fe(\y,\w)\fe(\x,\v)\,d\y d\v d\w.
\end{align*}

We summarize the above computation and obtain the following dynamics of $(\rhoe,\ue)$:
\begin{equation} \label{eq:EASeps}
	\begin{cases}
	\pa_t\rhoe+\grad_\x\cdot(\rhoe\ue)= 0,\\
    \pa_t(\rhoe\ue)+\grad_\x\cdot(\rhoe\ue\otimes\ue)+\grad_\x\cdot\Reps= \rhoe\A[\rhoe,\ue]+\Ge.
	\end{cases}
\end{equation}

Now, we take a formal limit $\eps\to0$. The leading order $\mathcal{O}(\eps^{-1})$ term in \eqref{eq:kineteps} is the local relaxation
\[\frac1\eps\grad_\v\cdot\big((\v-\ue)\fe\big)=0.\]
This implies that the limiting profile $f$ is mono-kinetic. More precisely, if 
\[\rhoe\to\rho,\quad \rhoe\ue\to\rho\u\]
in some appropriate sense, then we have
\[\fe(t,\x,\v)\to f(t,\x,\v)=\rho(t,\x)\delta_{\v=\u(t,\x)}.\]
Moreover, the mono-kinetic structure of $f$ implies that
\[\Reps\to0, \quad \Ge\to0.\]
Therefore, the limit quantities $(\rho,\u)$ solve the Euler-alignment system \eqref{eq:EAS}.

\subsection{The Euler equations with $p$-alignment}\label{sec:EASp}
For the Euler-alignment system with $p=2$, local and global well-posedness theories have been well-established for smooth solutions in Sobolev spaces $H^s(\Omega)\times H^{s+1}(\Omega)$, as discussed in, for example, \cite{tadmor2014critical}. 
The theory is based on the following non-conservative form of the Euler-alignment system:
\begin{equation} \label{eq:EASnon}
	\begin{cases}
	\pa_t\rho+\grad_\x\cdot(\rho\u)= 0,\\
    \pa_t\u+\u\cdot\grad_\x\u = \A[\rho,\u],
	\end{cases}
\end{equation}
The systems \eqref{eq:EAS} and \eqref{eq:EASnon} are equivalent if $\rho$ stays away from zero. Moreover, any smooth solution to \eqref{eq:EASnon} is also a solution to \eqref{eq:EAS}.

To be consistent with the conditions \eqref{eq:feinit} and \eqref{eq:fenormal} for the kinetic equation \eqref{eq:VAeps}, we assume initial data $(\rho^0,\u^0)$ satisfy
\begin{equation}\label{eq:rhouinit}
 \text{diam}(\text{supp}\rho^0)\leq\S^0<\infty, \quad \text{diam}(\text{Range}(\u^0))\leq\V^0<\infty,\quad \int_\Omega\rho^0(\x)\,d\x=1.
\end{equation}

One crucial aspect of the global well-posedness theory to \eqref{eq:EASnon} is the control of the Lipschitz bound on the velocity $[\u(t,\cdot)]_{Lip}$. Subsequently, the propagation of higher Sobolev norms follows from energy estimates.

However, for the case of general nonlinear alignment, obtaining smooth solutions is more challenging due to the non-smooth behavior of $\Phi$ near the origin. Here, we present a well-posedness theory for solutions in the space $(L^1\cap L^\infty)(\Omega)\times W^{1,\infty}(\Omega)$.

\begin{proposition}\label{prop:EAS}
Suppose the initial data $(\rho^0,\u^0)$ satisfy \eqref{eq:rhouinit} and
 \[(\rho^0,\u^0)\in (L^1\cap L^\infty)(\Omega)\times W^{1,\infty}(\Omega).\]
Then, there exists a time $T$ such that the system \eqref{eq:EASnon} with \eqref{eq:alignment} admits a unique strong solution
\[(\rho,\u)\in C([0,T),(L^1\cap L^\infty)(\Omega))\times C([0,T), W^{1,\infty}(\Omega)).\]
Moreover, the time span of the solution can be extended as long as
\begin{equation}\label{eq:gradubound}
 \|\grad_\x\u\|_{L^\infty([0,T)\times\Omega)}\leq M,
\end{equation}
where $M$ is a finite number.
\end{proposition}
\begin{proof}
From \eqref{eq:EASnon}, we obtain the dynamics of the velocity 
\begin{equation}\label{eq:udyn}
(\pa_t+\u\cdot\grad_\x)\u=\A[\rho,\u].	
\end{equation}
Applying gradient to the equation yields
\[(\pa_t+\u\cdot\grad_\x)\grad_\x\u=-(\grad_\x\u)^2+\grad_\x\A[\rho,\u].\]
We estimate the $p$-alignment as follows:
\begin{align*}
\big|\grad_\x\A[\rho,\u]\big|\leq&\,\int_\Omega \big|\grad\phi(\x-\y)\big|\cdot\big|\Phi(\u(\y)-\u(\x))\big|\cdot\rho(\y)\,d\y\\
&+\int_\Omega \phi(\x-\y)\cdot\Big|\grad\Phi(\u(\y)-\u(\x))\Big|\cdot\big|\grad\u(\x)\big|\cdot\rho(\y)\,d\y\\
\leq&\,[\phi]_{Lip}\cdot(\V^0)^{p-1}+\|\phi\|_{L^\infty}\cdot (p-1)(\V^0)^{p-2}\cdot\|\grad_\x\u\|_{L^\infty}.
\end{align*}
This leads to the estimate on $\grad_\x\u$:
\[\frac{d}{dt}\|\grad_\x\u(t,\cdot)\|_{L^\infty}\leq \|\grad_\x\u(t,\cdot)\|_{L^\infty}^2+\|\phi\|_{L^\infty}\cdot (p-1)(\V^0)^{p-2}\cdot\|\grad_\x\u\|_{L^\infty}+[\phi]_{Lip}\cdot(\V^0)^{p-1}.\]
Apply Cauchy-Lipschitz theorem, there exists a time $T>0$ such that $\|\grad_\x\u(t,\cdot)\|_{L^\infty}$ is bounded for any $t\in[0,T]$. Furthermore, \eqref{eq:gradubound} holds.

Note that $\|\u(t,\cdot)\|_{L^\infty}\leq\|\u^0\|_{L^\infty}$ by maximum principle (argued similarly as in Proposition \ref{prop:MP}). Consequently, we obtain an a priori bound on $\u(t)$ in $W^{1,\infty}(\Omega)$.

For the density, $\|\rho(t,\cdot)\|_{L^1}$	
  is conserved in time due to the conservation of mass. Given a Lipschitz velocity field \eqref{eq:gradubound}, $\|\rho(t,\cdot)\|_{L^\infty}$ has the a priori bound
  \[\|\rho(t,\cdot)\|_{L^\infty}\leq \|\rho^0\|_{L^\infty}e^{\int_0^t\|\grad_\x\u(s,\cdot)\|_{L^\infty}\,ds}\leq \|\rho^0\|_{L^\infty} e^{Mt},\]
for any $t\in[0,T]$. 

Next, we turn to prove uniqueness. Let $(\rho_1,\u_1)$ and $(\rho_2,\u_2)$ be two solutions to \eqref{eq:EASnon} with same initial data $(\rho^0,\u^0)$. Assume \eqref{eq:gradubound} holds for $\u_1$ and $\u_2$. Then the flow maps $\X_1(t,\x)$ and $\X_2(t,\x)$ with
\[\partial_t \X_i(t,\x) = \u_i(t,\X_i(t,\x)),\quad \X_i(0,\x)=\x,\quad i=1,2,\]
are well-defined on $[0,T]\times\Omega$. Let us denote
\[\delta_\X(t):=\sup_{\x\in\Omega}\big|\X_1(t,\x)-\X_2(t,\x)\big|,\quad
\delta_\u(t):=\sup_{\x\in\Omega}\big|\u_1(t,\X_1(t,\x))-\u_2(t,\X_2(t,\x))\big|.
\]
Clearly, we have for almost all $t\in[0,T]$,
\begin{equation}\label{eq:dx}
\frac{d}{dt}\delta_\X(t)\leq \delta_\u(t).
\end{equation}

Applying \eqref{eq:udyn} to $(\rho_i, \u_i)$ and evaluating at $(t, \X_i(t,\x))$, we obtain
\begin{align*}
 &\frac{d}{dt} \u_i(t,\X_i(t,\x)) = \A[\rho_i,\u_i](t,\X_i(t,\x))\\
 & = \int_\Omega \phi\big(\X_i(t,\x)-\widetilde{\y}\big)\,\Phi\Big(\u_i(t,\widetilde{\y})-\u_i(t,\X_i(t,\x))\Big)\rho_i(t,\widetilde{\y})\,d\widetilde{\y}\\
 & = \int_\Omega \phi\big(\X_i(t,\x)-\X_i(t,\y)\big)\,\Phi\Big(\u_i(t,\X_i(t,\y))-\u_i(t,\X_i(t,\x))\Big)\rho^0(\y)\,d\y.
\end{align*}
Here we change variable $\widetilde{\y}=\X_i(t,\y)$ and apply the $\rho$-equation in \eqref{eq:EASnon} to get
\[\rho_i(t,\widetilde{\y})\,d\widetilde{\y}=\rho^0(\y)\,d\y.\]

For simplicity, we suppress the $t$-dependence and use the following shortcut notations:
\[\X_i:=\X_i(t,\x), \quad \Y_i:=\X_i(t,\y).\]

Compute the difference
\begin{align}
& \frac{d}{dt}\big(\u_1(\X_1)- \u_2(\X_2)\big)= \int_\Omega \Big(\phi(\X_1-\Y_1)-\phi(\X_2-\Y_2)\Big)\Phi\big(\u_1(\Y_1)-\u_1(\X_1)\big)\rho^0(\y)\,d\y\label{eq:du12}\\
& + \int_\Omega \phi(\X_2-\Y_2)\Big(\Phi\big(\u_1(\Y_1)-\u_1(\X_1)\big)-\Phi\big(\u_2(\Y_2)-\u_2(\X_2)\big)\Big)\rho^0(\y)\,d\y =: H_1+H_2.\nonumber
\end{align}
We now estimate term by term. For $H_1$,
\begin{align*}
|H_1| & \leq \int_\Omega [\phi]_{Lip}|(\X_1-\X_2)-(\Y_1-\Y_2)|\cdot \Phi\big(\u_1(\Y_1)-\u_1(\X_1)\big)\cdot\rho^0(\y)\,d\y\\
&\leq [\phi]_{Lip}\cdot 2\delta_\X\cdot (\V^0)^{p-1}.
\end{align*}
For $H_2$, we have
\begin{align*}
|H_2| & = \left|\int_\Omega \phi(\X_2-\Y_2)\cdot\grad\Phi(\xi)\Big(\big(\u_1(\Y_1)-\u_2(\Y_2)\big)-\big(\u_1(\X_1)-\u_2(\X_2)\big)\Big)\cdot\rho^0(\y)\,d\y\right|\\
&\leq \|\phi\|_{L^\infty}\cdot (p-1) (\V^0)^{p-2}\cdot 2\delta_\u.
\end{align*}
Here $\xi$ lies in between $\u_1(\Y_1)-\u_1(\X_1)$ and $\u_2(\Y_2)-\u_2(\X_2)$. From maximum principle analogous to \eqref{eq:MP} (see \cite{black2024asymptotic}), we have 
$|\u_i(\Y_i)-\u_i(\X_i)|\leq \V^0$, for $i=1,2$. Therefore, $|\xi|\leq\V^0$, leading to the last inequality above.

Applying both estimates to \eqref{eq:du12}, we deduce
\begin{equation}\label{eq:du}
\frac{d}{dt}\delta_\u(t)\leq C \big(\delta_\X(t)+\delta_\u(t)\big),	
\end{equation}
for almost all $t\in[0,T]$, where the constant $C$ depends on $\phi, \V^0$, and $p$.

We put together \eqref{eq:dx} and \eqref{eq:du}, and obtain
\[\frac{d}{dt}\big(\delta_\X(t)+\delta_\u(t)\big)\leq (C+1) \big(\delta_\X(t)+\delta_\u(t)\big).\]
Since $\delta_\X(0)=\delta_\u(0)=0$, we conclude with
\[\delta_\X(t)+\delta_\u(t)\leq \big(\delta_\X(0)+\delta_\u(0)\big)e^{(C+1)t}=0,\]
namely $\delta_\X(t)=\delta_\u(t)=0$ for any $t\in[0,T]$. Therefore,
\[\X_1(t,\x)=\X_2(t,\x)=:\X,\]
and we conclude with uniqueness:
\begin{align*}
\u_1(t,\X) &= \u_1(t,\X_1(t,\x)) = \u_2(t,\X_2(t,\x)) = \u_2(t,\X),\\
\rho_1(t,\X) &= \rho^0(\x)e^{\int_0^t\grad_\x\cdot\u_1(\tau,\X_1(\tau,\x)) d\tau} = \rho^0(\x)e^{\int_0^t\grad_\x\cdot\u_2(\tau,\X_2(\tau,\x)) d\tau}=\rho_2(t,\X),
\end{align*}
for any $t\in[0,T]$ and $\X\in\Omega$. 
\end{proof}

\subsection{Statement of the main result}\label{sec:statement}
Our main goal is to establish a rigorous derivation of the hydrodynamic limit. We consider the following well-prepared initial data $\fe^0$ satisfying \eqref{eq:feinit} where $(\S^0,\V^0)$ are independent of $\eps$. Moreover, $\fe^0$ is close to the initial data $(\rho^0,\u^0)$ of the limiting system \eqref{eq:EASnon}, in the sense
\begin{equation}\label{eq:initapp}
 W_1(\fe^0,f^0)<\eps,	
\end{equation}
where $f^0$ is defined as
\[f^0(\x,\v)=\rho^0(\x)\delta_{\v=\u^0(\x)},\]
and $W_1$ is the 1-Wasserstein metric. It can be defined through the dual representation 
\[ 
W_1(f,g)=\sup_{[\varphi]_{Lip}\leq 1} \int_X \varphi(x)(f(x)-g(x))\,dx
\]
where $f,g$ are arbitrary real-valued functions on $X$. In our context, $X=\Omega\times\R^d$.


We now state our main result on the hydrodynamic limit.
\begin{theorem}\label{thm:main}
 Assume the initial data $\fe^0$ and $(\rho^0,\u^0)$ satisfy \eqref{eq:feinit}-\eqref{eq:fenormal}, \eqref{eq:rhouinit}, and \eqref{eq:initapp}.
 Let $\fe$ be a weak solution to \eqref{eq:VAeps}, and $(\rho,\u)$ be a strong solution to \eqref{eq:EASnon} up to time $T$. 
 Denote
 \[f(t,\x,\v)=\rho(t,\x)\delta_{\v=u(t,\x)}.\]
  Then, we have
 \begin{equation}\label{eq:convergence}
  \fe(t,\x,\v)\rightharpoonup \rho(t,\x)\delta_{\v=u(t,\x)}\quad\text{in }\mathcal{M}((0,T)\times\Omega\times\R^d),
 \end{equation}
 where $\mathcal{M}((0,T)\times\Omega\times\R^d)$ is the space of nonnegative Radon measures on $(0,T)\times\Omega\times\R^d$.
\end{theorem}

More quantitative estimates to the limit \eqref{eq:convergence} will be presented in \eqref{eq:rhouconv} and \eqref{eq:rhouconvrate}.

\section{Rigorous derivation} \label{sec:Main}
In this section, we present the proof of our main theorem regarding the rigorous hydrodynamic limits, as outlined in Theorem \ref{thm:main}. When the velocity alignment is linear ($p=2$), a framework has been established in \cite{figalli2018rigorous}. Our approach extends this framework to accommodate situations where the velocity alignment is nonlinear ($p>2$). It is worth noting that we must establish additional controls to account for discrepancies generated by the nonlinearity, as detailed in Sections \ref{sec:energy} and \ref{sec:I4}.

\subsection{Relative entropy method}\label{sec:RE}
Our principal approach for rigorously establishing the hydrodynamic limit relies on the \emph{relative entropy method}. We closely adhere to the framework outlined in \cite{figalli2018rigorous} and focus our efforts on analyzing the following quantity:
\begin{equation}\label{eq:eta}
 \etae(t) = \frac12\int_\Omega\rhoe(t,x) |\ue(t,x)-\u(t,\x)|^2\,d\x.
\end{equation}

Let us remark the meaning of $\etae$.
Let $U=(\rho,\mathbf{m})=(\rho,\rho\u)$. A convex entropy on $U$ is defined as 
 \[\eta(U)=\eta(\rho,\mathbf{m}):=\frac{|\mathbf{m}|^2}{2\rho}=\frac{\rho|\u|^2}2.\]
 Then, we may define the relative entropy
 \begin{align*}
   \eta(U_\eps|U)=&\,\eta(U_\eps)-\eta(U)-D\eta(U)\cdot(U_\eps-U)\\
	=&\,\frac{\rhoe|\ue|^2}{2}-\frac{\rho|\u|^2}{2}-\left(-\frac{|\u|^2}{2}(\rhoe-\rho)+\u\cdot(\rhoe\ue-\rho\u)\right)=\frac12~\rhoe |\ue-\u|^2.	
 \end{align*}
Finally, the quantity $\etae$ defined in \eqref{eq:eta} is the spatial integration of the relative entropy $\eta(U_\eps|U)$.

We investigate the evolution of $\etae$ through the following calculation:
\[
\frac{d}{dt}\etae=\frac{d}{dt}\int_{\Omega} \left(\frac{\rhoe |\ue|^2}{2}-\rhoe\ue \cdot \u+\frac{\rhoe|\u|^2}{2}\right)\,d\x=\frac{d}{dt}E_\eps+\I+\II,
\]
where $E_\eps$ is the \emph{macroscopic energy}
\[
E_\eps=\frac12\int_{\Omega} \rhoe |\ue|^2\,d\x.
\]

In particular, for $\I$ we have
\begin{align*}
\I=&\int_\Omega \big(-\pa_t(\rhoe\ue)\cdot\u-\rhoe\ue\cdot\pa_t\u\big)\,d\x\\
=&\int_\Omega \left(\grad_\x\cdot(\rhoe\ue\otimes\ue+\Reps)-\rhoe\A[\rhoe,\ue]-\Ge\right)\cdot\u\,d\x\\
&+\int_\Omega\rhoe\ue\cdot\big( \u\cdot\grad_\x\u-\A[\rho,\u]\big)\,d\x\\
=&\int_\Omega \rhoe\ue\otimes(\ue-\u):\grad_\x\u\,d\x-\int_\Omega \rhoe\big(\u\cdot\A[\rhoe,\ue]+\ue\cdot\A[\rho,\u]\big)\,d\x\\
&-\int_\Omega \grad_\x\u:\Reps\,d\x-\int_\Omega \Ge\cdot\u\,d\x\\
=&\,\I_1+\I_2+\I_3+\I_4.
\end{align*}
Similarly, for $\II$ we have
\begin{align*}
\II=&\,\frac12\int_\Omega \pa_t\rhoe\,|\u|^2\,d\x+\int_\Omega \rhoe\u\cdot\pa_t\u\,d\x\\
=&-\frac12\int_\Omega \grad_\x\cdot(\rhoe\ue)~|\u|^2\,d\x+\int_\Omega\rhoe\u\cdot\big( -\u\cdot\grad_\x\u+\A[\rho,\u]\big)\,d\x\\
=&-\int_\Omega \rhoe\u\otimes(\ue-\u):\grad_\x\u\,d\x+\int_\Omega\rhoe\u\cdot\A[\rho,\u]\,d\x\\
=&\,\II_1+\II_2.
\end{align*}

Next, we estimate all terms in $\I$ and $\II$. We start with two straightforward bounds
\[
|\I_1+\II_1|=\left|\int_\Omega\rhoe(\ue-\u)\otimes(\ue-\u):\grad_\x\u\,d\x\right|\leq\|\grad_\x\u\|_{L^\infty}\,\etae,
\]
and
\[
|\I_3|=\left|\int_\Omega\grad_x\u:\Reps\,d\x\right|\leq\|\grad_x\u\|_{L^\infty}\int_{\Omega\times\R^d} \fe(\x,\v)|\v -\ue|^2\,d\x d\v=\|\grad_\x\u\|_{L^\infty}\,\Dee,
\]
where we recall the definition of $\Dee$ in \eqref{eq:D2}.

Then, we focus on the term
\begin{equation}\label{eq:alignestJ}
J:=\I_2+\II_2=\int_\Omega\rhoe\big(-\u\cdot\A[\rhoe,\ue]+(\u-\ue)\cdot\A[\rho,\u]\big)\,d\x.	
\end{equation}
Start with the first term in \eqref{eq:alignestJ} and get
\begin{align*}
&\int_\Omega-\rhoe\u\cdot\A[\rhoe,\ue]\,d\x=
\int_{\Omega^2}-\rhoe(\x)\rhoe(\y)\phi(\x-\y)\u(\x)\cdot\Phi(\ue(\y)-\ue(\x))\,d\x d\y\\
&=\frac12\int_{\Omega^2}\rhoe(\x)\rhoe(\y)\phi(\x-\y)(\u(\y)-\u(\x))\cdot\Phi(\ue(\y)-\ue(\x))\,d\x d\y\\
&=\frac12\int_{\Omega^2}\rhoe(\x)\rhoe(\y)\phi(\x-\y)(\ue(\y)-\ue(\x))\cdot\Phi(\ue(\y)-\ue(\x))\,d\x d\y\\
&+\frac12\int_{\Omega^2}\rhoe(\x)\rhoe(\y)\phi(\x-\y)(\u(\y)-\u(\x)-\ue(\y)+\ue(\x))\cdot\Phi(\ue(\y)-\ue(\x))\,d\x d\y\\
&=:D_\eps+J_1.
\end{align*}
Here, we symmetrize $\x$ and $\y$ in the second equality, followed by splitting the quantity into two parts. In particular, $D_\eps$ is the \emph{macroscopic enstrophy}
\[D_\eps = \frac12\int_{\Omega^2}\rhoe(\x)\rhoe(\y)\phi(\x-\y)|\ue(\y)-\ue(\x)|^p\,d\x d\y.\]
We will control $D_\eps$ later by the kinetic enstrophy $\De$.

Now we work on the second term in \eqref{eq:alignestJ}. Split the term into two parts:
\begin{align*}
&\int_\Omega\rhoe(\u-\ue)\cdot\A[\rho,\u]\,d\x\\
&=\int_{\Omega^2}\rhoe(\x)\rho(\y)\phi(\x-\y)(\u(\x)-\ue(\x))\cdot\Phi(\u(\y)-\u(\x))\,d\x d\y\\
&=\int_{\Omega^2}\rhoe(\x)\rhoe(\y)\phi(\x-\y)(\u(\x)-\ue(\x))\cdot\Phi(\u(\y)-\u(\x))\,d\x d\y\\
&+\int_{\Omega^2}\rhoe(\x)(\rho(\y)-\rhoe(\y))\phi(\x-\y)(\u(\x)-\ue(\x))\cdot\Phi(\u(\y)-\u(\x))\,d\x d\y\\
&=:J_2+J_3.
\end{align*}
We further symmetrize $\x$ and $\y$ in $J_2$ and obtain
\[
J_2=\frac12\int_{\Omega^2}\rhoe(\x)\rhoe(\y)\phi(\x-\y)(\u(\x)-\u(\y)-\ue(\x)+\ue(\y))\cdot\Phi(\u(\y)-\u(\x))\,d\x d\y.
\]
Combing $J_1$ and $J_2$, we get
\begin{align*}
J_1+J_2=&\,	\frac12\int_{\Omega^2}\rhoe(\x)\rhoe(\y)\phi(\x-\y)\big(\u(\y)-\u(\x)-\ue(\y)+\ue(\x)\big)\\&\qquad\cdot\big(\Phi(\ue(\y)-\ue(\x))-\Phi(\u(\y)-\u(\x))\big)\,d\x d\y.
\end{align*}
Observe that since $\Phi$ is monotone increasing, we have
\[(\z_1-\z_2)\cdot(\Phi(\z_1)-\Phi(\z_2))\geq0,\quad\forall~\z_1,\z_2.\]
This yields 
\[J_1+J_2\leq0.\]

For the remaining term $J_3$:
\[J_3=\int_\Omega\rhoe(\x)(\u(\x)-\ue(\x))\cdot\left[\int_\Omega(\rho(\y)-\rhoe(\y))\phi(\x-\y)\Phi(\u(\y)-\u(\x))\,d\y\right]d\x,\]
we obtain the point-wise bound on the inner integral
\begin{align*}
&\int_\Omega (\rho(\y)-\rhoe(\y))\phi(\x-\y)\Phi(\u(\y)-\u(\x))\,d\y	\\
\leq&\,W_1(\rho,\rhoe)\cdot\Big[\phi(\x-\cdot)\,\Phi(\u(\cdot)-\u(\x))\Big]_{Lip}\\
\leq&\,W_1(\rho,\rhoe)\Big([\phi]_{Lip}(\V^0)^{p-1}+\|\phi\|_{L^\infty}(p-1)(\V^0)^{p-2}\|\grad_\x\u\|_{L^\infty}\Big)\\
\leq&\, C(1+\|\grad_\x\u\|_{L^\infty})\,W_1(\rho,\rhoe),
\end{align*}
for any $\x\in\Omega$. Here, we have used the maximum principle \eqref{eq:MP}.

We then apply H\"older inequality and obtain 
\[|J_3|\leq C(1+\|\grad_\x\u\|_{L^\infty})\,W_1(\rho,\rhoe)\cdot\sqrt{\etae}\leq C(1+\|\grad_\x\u\|_{L^\infty})\big(\etae+W_1^2(\rho,\rhoe)\big).\]

Collecting all the estimates, we conclude with
\begin{equation}\label{eq:etaest}
\frac{d}{dt}\etae\leq\frac{d}{dt}E_\eps+D_\eps+\|\grad_x\u\|_{L^\infty}\Dee+C(1+\|\grad_\x\u\|_{L^\infty})\big(\etae+ W_1^2(\rho,\rhoe)\big) +\I_4.
\end{equation}

\subsection{The control of macroscopic energy and enstrophy}\label{sec:energy}
In this part, we aim to obtain a bound on 
\[\frac{d}{dt}E_\eps+D_\eps\]
that appears on the right hand side of \eqref{eq:etaest}. We compare $E_\eps$ and $D_\eps$ with the kinetic energy $\Ee$ and enstrophy $\De$, and apply \eqref{eq:energydecay} to obtain
\begin{equation}\label{eq:EDest}
\frac{d}{dt}E_\eps+D_\eps\leq \frac{d}{dt}(E_\eps-\Ee)+(D_\eps-\De)-\frac1\eps\Dee.	
\end{equation}

Now, let us control the differences between kinetic and macroscopic energies and enstrophies.

\begin{lemma}
The following inequalities hold:
\begin{align}
E_\eps(t)&\leq\Ee(t),\label{eq:EE}\\
D_\eps(t)&\leq\De(t)+|\Delta_\eps(t)|,\label{eq:DD}
\end{align}
where the discrepancy $\Delta_\eps(t)$ is defined as
\begin{align}\label{eq:Deltaeps}
 \Delta_\eps(t):=\frac{1}{2}\int_{\Omega^2\times\R^{2d}} \phi(\x-\y)\,(|\w-\v|^{p-2}-|\ue(\y)-\ue(\x)|^{p-2})\,|\w-\v|^2\\
\fe(t,\x,\v)\fe(t,\y,\w)\,d\x d\y d\v d\w.	\nonumber
\end{align}
\end{lemma}
\begin{proof}
The first inequality \eqref{eq:EE} follows directly from the Cauchy-Schwarz inequality
\begin{equation}\label{eq:CSineq}
\rhoe|\ue|^2=\frac{\left(\int_{\R^d}\v\fe\,d\v\right)^2}{\int_{\R^d}\fe\,d\v}\leq\int_{\R^d}|\v|^2\fe\,d\v.	
\end{equation}
For the second inequality \eqref{eq:DD}, we decompose $\De$ into two parts:
\[
\De=\frac{1}{2}\int_{\Omega^2\times\R^{2d}} \phi(\x-\y)|\ue(\y)-\ue(\x)|^{p-2}|\w-\v|^2\fe(\x,\v)\fe(\y,\w)\,d\x d\y d\v d\w+\Delta_\eps.
\]
For the first part, we apply \eqref{eq:CSineq} and obtain
\begin{align*}
 &\int_{\R^{2d}}|\w-\v|^2\fe(\x,\v)\fe(\y,\w)\,d\v d\w=\int_{\R^{2d}}(|\w|^2-2\w\cdot\v+|\v|^2)\fe(\x,\v)\fe(\y,\w)\,d\v d\w\\
 &=\rhoe(\x)\int_{\R^d}|\w|^2\fe(\y,\w)\, d\w-2\rhoe(\x)\ue(\x)\cdot\rhoe(\y)\ue(\y)+\rhoe(\y)\int_{\R^d}|\v|^2\fe(\x,\v)\, d\v\\
 &\geq \rhoe(\x)\rhoe(\y)\left(|\ue(\y)|^2-2\ue(\x)\cdot\ue(\y)+|\ue(\x)|^2\right)=\rhoe(\x)\rhoe(\y)|\ue(\x)-\ue(\y)|^2.
\end{align*}
This leads to the bound
\[\frac{1}{2}\int_{\Omega^2\times\R^{2d}} \phi(\x-\y)|\ue(\y)-\ue(\x)|^{p-2}|\w-\v|^2\fe(\x,\v)\fe(\y,\w)\,d\x d\y d\v d\w\geq D_\eps.\]
The inequality \eqref{eq:DD} follows as a direct consequence.
\end{proof}

When $p=2$, the discrepancy $\Delta_\eps(t)=0$. However, with the nonlinear alignment $p>2$, $\Delta_\eps$ does not vanish unless $\fe$ is mono-kinetic. Therefore, we will use the kinetic enstrophy $\Dee$ to control the discrepancy.

\begin{lemma}\label{lem:pdiff}
 Let $a,b\in[0,R]$. The following inequalities hold:
 \begin{align*}
   &\left|a^{p-2}-b^{p-2}\right|\leq |a-b|^{p-2},\quad \text{for } 2<p\leq3 ,\quad \\
   &\left|a^{p-2}-b^{p-2}\right|\leq (p-2)R^{p-3}|a-b|,\quad \text{for } p>3.
\end{align*}
\end{lemma}
\begin{proof}
For the first inequality, we assume $b\leq a$ without loss of generality. If $b=0$, the equality holds trivially. If $b>0$, define $z=a/b\in[1,\infty)$. The inequality is equivalent to
\[g(z):=z^{p-2}-1-(z-1)^{p-2}\leq0.\]
One can easily verify $g(1)=0$ and $g'(z)\leq0$ for $z\geq1$. This leads to the desired inequality.

The second inequality is a direct application of the mean value theorem.	
\end{proof}

Now we apply Lemma \ref{lem:pdiff} with $a=|\w-\v|$ and $b=|\ue(\y)-\ue(\x)|$. Let
\[q = \min\{p-2,1\}\leq1,\quad \text{and}\quad c_p=\begin{cases}
0&p=2,\\1&2<p\leq3,\\(p-2)(\V^0)^{p-3}&p>3.
\end{cases}
\]
We have
\begin{align}\label{eq:deltaest}
   &\Big||\w-\v|^{p-2}-|\ue(\y)-\ue(\x)|^{p-2}\Big| \leq c_p\Big||\w-\v|-|\ue(\y)-\ue(\x)|\Big|^{q}\\
    &\leq c_p\Big(|\v-\ue(\x)|+|\w-\ue(\y)|\Big)^{q}\leq c_p\Big(|\v-\ue(\x)|^{q}+|\w-\ue(\y)|^q\Big).\nonumber
\end{align}
Note that we have used triangle inequality in the second inequality, and concavity of the function $x^q$ in the last inequality.
\begin{remark}\label{rmk:Phi}
 For general nonlinearity \eqref{eq:Phigeneral}, Lemma \ref{lem:pdiff} can be replaced by
 \[|h(a)-h(b)|\leq C_R|a-b|^q,\]
 where $C_R$ is the H\"older (or Lipschitz) coefficient on $h$ on $[0,R]$. Then \eqref{eq:deltaest} follows with $c_p = C_{\V^0}$.
\end{remark}

Utilizing the estimate \eqref{eq:deltaest} and H\"older inequality, we can bound $\Delta_\eps$ as follows:
\begin{align}
     |\Delta_\eps|&\leq c_p\int_{\Omega^2\times\R^{2d}}\phi(\x-\y)|\v -\ue(\x)|^{q}|\w-\v|^2\fe(\x,\v)\fe(\y,\w)\,d\x d\y d\v d\w \nonumber\\
     &\leq c_p\left(\int_{\Omega\times\R^d}|\v-\ue(\x)|^2\fe(\x,\v)\,d\x d\v\right)^{\frac{q}{2}}\cdot\nonumber\\
     &\qquad\left(\int_{\Omega\times\R^d}\left(\int_{\Omega\times\R^d}\phi(\x-\y)|\v -\w|^2\fe(\y,\w)\,d\y d\w\right)^{\frac{2}{2-q}}\fe(\x,\v)\,d\x d\v \right)^{\frac{2-q}{2}}\nonumber\\
     &\leq c_p\cdot\Dee^{\frac{q}{2}}\cdot\|\phi\|_{L^\infty}\cdot(\V^0)^2\leq \frac{C_p}{2}\,\Dee^{\frac{q}{2}},\label{eq:deltabound}
\end{align}
where we defined the constant $C_p=8c_p\|\phi\|_{L^\infty}\|\u^0\|_{L^\infty}^2$.

\subsection{The control of the discrepancy $\I_4$}\label{sec:I4}
When $p=2$, the term $\Ge$ vanishes. Hence we have $\I_4=0$. With the nonlinear alignment $p>2$, $\Ge$ does not vanish unless $\fe$ is mono-kinetic. Therefore, we may treat $\I_4$ similarly as the discrepancy $\Delta_\eps$.
\begin{align}
|\I_4|&\leq	\|\u\|_{L^\infty}\int_{\Omega^2\times\R^{2d}}\phi(\x-\y)\,\Big||\w-\v|^{p-2}-|\ue(\y)-\ue(\x)|^{p-2}\Big|\,|\w-\v|\nonumber\\
&\hspace{2.5in}\fe(\x,\v)\fe(\y,\w)\,d\x d\y d\v d\w\nonumber\\
&\leq 2c_p\|\u\|_{L^\infty}\int_{\Omega^2\times\R^{2d}}\phi(\x-\y)\,|\v-\ue(\x)|^{q}\,|\w-\v|\fe(\x,\v)\fe(\y,\w)\,d\x d\y d\v d\w\nonumber\\
&\leq 2c_p\|\phi\|_{L^\infty}\V^0\|\u^0\|_{L^\infty}\cdot \Dee^{\frac{q}{2}}\leq\frac{C_p}{2}\Dee^{\frac{q}{2}}.\label{eq:I4}
\end{align}

Collecting the estimates \eqref{eq:EDest}, \eqref{eq:DD}, \eqref{eq:deltabound} and \eqref{eq:I4}, the bound \eqref{eq:etaest} becomes
\begin{align*}
\frac{d}{dt}\etae\leq&\,\left(-\frac{1}{\eps}\Dee+\|\grad_\x\u\|_{L^\infty}\Dee+C_p\Dee^{\frac{q}{2}}\right)+\frac{d}{dt}(E_\eps-\Ee)\\
&\quad+C(1+\|\grad_\x\u\|_{L^\infty})\big(\etae+ W_1^2(\rho,\rhoe)\big).	
\end{align*}
Integrate in $[0,t]$ and apply \eqref{eq:EE}. We end up with
\begin{align}
\etae(t)\leq&\,\big(\etae(0)+\Ee(0)-E_\eps(0)\big)\nonumber\\
&+\int_0^t\left(-\frac{1}{\eps}\Dee(s)+\|\grad_\x\u(s,\cdot)\|_{L^\infty}\Dee(s)+C_p\Dee^{\frac{q}{2}}(s)\right)\,ds\nonumber\\
&+C\int_0^t(1+\|\grad_\x\u(s,\cdot)\|_{L^\infty})\Big(\etae(s)+ W_1^2(\rho(s,\cdot),\rhoe(s,\cdot))\Big)\,ds.\label{eq:etaestint}
\end{align}

\subsection{The control of the relative entropy}\label{sec:etae}
Now, we proceed to demonstrate the convergence of the relative entropy $\etae(t)$ to zero as $\eps\to0$. To achieve this, we will estimate all three terms on the right-hand side of \eqref{eq:etaestint}.

For the first term in \eqref{eq:etaestint}, which concerns the initial data, we contend that it can be effectively controlled by the Wasserstein-1 distance $W_1(\fe^0, f^0)$. As a result, we can establish its smallness, in accordance with the assertion in \eqref{eq:initapp}. More precisely, we have the following bounds:
\begin{align}
\etae(0)+\Ee(0)-E_\eps(0)&=\frac12\int_\Omega\left(\rhoe^0|\ue^0-\u^0|^2-\rhoe^0|\ue^0|^2\right)\,d\x+\frac12\int_{\Omega\times\R^d}|\v|^2\fe^0\,d\x d\v\nonumber\\
&=\frac12\int_{\Omega\times\R^d}|\v-\u^0(\x)|^2\fe^0(\x,\v)\,d\x d\v\nonumber\\
&=\frac12\int_{\Omega\times\R^d}|\v-\u^0(\x)|^2(\fe^0(\x,\v)-f^0(\x,\v))\,d\x d\v\nonumber\\
&\leq \frac12\left[|\v-\u^0(\x)|^2\right]_{Lip_{(\x,\v)}}W_1(\fe^0,f^0)\leq C\eps.\label{eq:eta1}
\end{align}
Note that $\left[|\v-\u^0(\x)|^2\right]_{Lip_{(\x,\v)}}$ is bounded as $\fe^0$ is compactly supported, and $\u^0$ is Lipschitz. The constant $C$ may be taken as $C=(1+[\u^0]_{Lip})\V^0$.

Next, we discuss the second term in \eqref{eq:etaestint}. When $p=2$, we have $C_2=0$ and there is no discrepancy. We take $\eps$ small enough with $\eps<\frac1M$. Recall the a priori bound \eqref{eq:gradubound}. It implies that the second term in \eqref{eq:etaestint} is negative.  

 When $p>2$, we need to obtain an additional control to the discrepancy $C_p\Dee^{\frac{q}{2}}$.
Consider the function
\[F(x) = -ax+bx^\gamma\]
with $a,b>0$ and $\gamma\in(0,1)$. One can easily obtain the maximum
\[\max_{x\geq0}F(x)=F(x_*)=\left(\gamma^{\frac{1}{1-\gamma}}+\gamma^{\frac{\gamma}{1-\gamma}}\right)a^{-\frac{\gamma}{1-\gamma}}b^{\frac{1}{1-\gamma}},\quad\text{where}\,\,x_*=\gamma^{\frac{1}{1-\gamma}}a^{-\frac{1}{1-\gamma}}b^{\frac{1}{1-\gamma}}.\]
Take $\eps$ small enough such that $\eps<\frac{1}{2M}$. Apply the above estimate with 
\[a=\frac{1}{\eps}-\|\grad_\x\u(t,\cdot)\|_{L^\infty}>\frac{1}{2\eps},\quad b=C_p\quad \text{and} \quad \gamma=\frac{q}{2}.\]
We have
\begin{align}
-\frac{1}{\eps}\Dee(t)&+\|\grad_\x\u(t,\cdot)\|_{L^\infty}\Dee(t)+C_p\Dee^{\frac{q}{2}}(t)\nonumber\\ &\leq
\left(\left(\tfrac{q}{2}\right)^{\frac{2}{2-q}}+\left(\tfrac{q}{2}\right)^{\frac{q}{2-q}}\right)(2\eps)^{\frac{q}{2-q}}C_p^{\frac{2}{2-q}}\leq C\eps^{\frac{q}{2-q}},\label{eq:eta2}
\end{align}
where the constant $C$ depends only on $p$.

Applying \eqref{eq:eta1} and \eqref{eq:eta2} to \eqref{eq:etaestint}, we deduce
\begin{align}
\etae(t)\leq&\,C\eps+C\eps^{\frac{q}{2-q}}t+C\int_0^t\Big(\etae(s)+ W_1^2(\rho(s,\cdot),\rhoe(s,\cdot))\Big)\,ds.\label{eq:etaestint2}
\end{align}
The constant $C$ depends on initial data, the parameter $p$, and the a priori bound $M$ (see \eqref{eq:gradubound}).

To close the argument, we state the following control on $W_1(\rho,\rhoe)$ by the relative entropy $\etae$. 
\begin{lemma}\label{lem:W1}
There exist a constant $C=C(T,M)>0$, such that
\begin{equation}\label{eq:W1bound}
W_1^2(\rho(t,\cdot),\rhoe(t,\cdot))\leq C\left(W_1^2(\rho^0,\rhoe^0)+\int_0^t\etae(s)\,ds\right),	
\end{equation}
for any $t\in[0,T]$.
\end{lemma}
Some versions of Lemma \ref{lem:W1} have been developed in e.g. \cite[Lemma 5.2]{figalli2018rigorous} (control $W_2$ distance by $\etae$), and \cite[Lemma 5.2]{shvydkoy2021dynamics} (control $W_1$ distance by kinetic relative entropy). We include a proof of the Lemma here for self-consistency.

\begin{proof}[Proof of Lemma \ref{lem:W1}]
Consider the flow maps
$\Xe$ and $\X$ defined as
\[\begin{cases}
\pa_t\Xe(t,\x)=\ue(t,\Xe(t,\x)),\\
\Xe(0,\x)=\x,	
\end{cases}
\quad\text{and}\quad
\begin{cases}
\pa_t\X(t,\x)=\u(t,\X(t,\x)),\\
\X(0,\x)=\x.
\end{cases}
\]
The solutions $\rhoe$ and $\rho$ of the continuity equations can be viewed as the push-forward of the initial measure
\[\rhoe(t)=\Xe(t)_\#\rhoe^0,\quad\text{and}\quad
\rho(t)=\X(t)_\#\rho^0.\]
Note that since $\ue$ is not necessarily Lipschitz, $\Xe$ is not uniquely defined. The push-forward relation should be realized in the probabilistic sense, see e.g. \cite[Proposition 3.3]{figalli2018rigorous}. Namely, there exists a probability measure $\eta_\eps$, defined on $\Gamma_T\times\Omega$, where $\Gamma_T$ denotes the space of absolutely continuous curves from $[0,T]$ to $\Omega$, such that

(i) $\eta_\eps$ is concentrated on the set of pairs $(\gamma,\x)$ such that $\gamma$ is a solution of
\begin{equation}\label{eq:gamma}
\gamma'(t)=\u_\eps(t,\gamma(t)),\quad \forall~a.e.~t\in(0,T),\text{ with } \gamma(0)=\x,	
\end{equation}

(ii) $\rhoe$ satisfies
\[\int_\Omega \varphi(\x)\rhoe(t,\x)\,d\x=\int_{\Gamma_T\times\Omega}\varphi(\gamma(t))\,d\eta(\gamma,\x)=\int_{\Gamma_T\times\Omega}\varphi(\gamma(t))\,\eta_\eps(d\gamma,\x)\rhoe^0(\x)\,d\x,\]
for any test function $\varphi\in C_c^0(\Omega)$ and $t\in[0,T]$. Here $\eta_\eps(d\gamma,\x)$ is the (marginal) probability distribution on  functions in $\Gamma_T$ that solves \eqref{eq:gamma} with a given $\x$.

Let us define another measure
\[\tilde{\rho}_\eps(t)=\X(t)_\#\rhoe^0,\]
and decompose
\begin{equation}\label{eq:W1decomp}
W_1(\rho(t,\cdot),\rhoe(t,\cdot))\leq W_1(\rho(t,\cdot),\rhot(t,\cdot))+W_1(\rhot(t,\cdot),\rhoe(t,\cdot)).	
\end{equation}
For the first part, $\rho$ and $\tilde{\rho}_\eps$ shares the same flow. We obtain
\begin{align}
W_1(\rho(t,\cdot),\rhot(t,\cdot))=&\,\sup_{[g]_{Lip}\leq1}\int_\Omega g(\x)(\rho(t,\x)-\rhot(t,\x)	)\,d\x\nonumber\\
=&\,\sup_{[g]_{Lip}\leq1}\int_\Omega g(\X(t,\x))(\rho^0(\x)-\rhoe^0(\x)	)\,d\x\nonumber\\
\leq&\,\|\grad\X(t,\cdot)\|_{L^\infty}\,W_1(\rho^0,\rhoe^0)\leq e^{Mt}W_1(\rho^0,\rhoe^0).\label{eq:W11}
\end{align}
For the second part, $\rhot$ and $\rhoe$ shares the same initial data. We have
\begin{align*}
W_1(\rhot(t,\cdot),\rhoe(t,\cdot))=&\,\sup_{[g]_{Lip}\leq1}\int_\Omega g(\x)(\rhot(t,\x)-\rhoe(t,\x))\,d\x\\
=&\,\sup_{[g]_{Lip}\leq1}\int_\Omega \left(g(\X(t,\x))-\int_{\Gamma_T}g(\gamma(t))\eta_\eps(d\gamma,\x)\right)\rhoe^0(\x)\,d\x\\
\leq&\,\int_{\Gamma_T\times\Omega} \big|\X(t,\x)-\gamma(t)\big|\,\eta_\eps(d\gamma,\x)\cdot\rhoe^0(\x)\,d\x=:B(t).
\end{align*}
Clearly, $B(t)$ is continuous, and differentiable for almost every $t\in[0,T]$. Compute the time derivative of $B(t)$ and get
\begin{align*}
B'(t)\leq&\,\int_{\Gamma_T\times\Omega} \big|\u(t,\X(t,\x))-\ue(t,\gamma(t))\big|\,\eta_\eps(d\gamma,\x)\cdot\rhoe^0(\x)\,d\x\\
\leq&\,\int_{\Gamma_T\times\Omega} \big|\u(t,\X(t,\x))-\u(t,\gamma(t))\big|\,\eta_\eps(d\gamma,\x)\cdot\rhoe^0(\x)\,d\x\\
&+\int_{\Gamma_T\times\Omega} \big|\u(t,\gamma(t))-\ue(t,\gamma(t))\big|\,\eta_\eps(d\gamma,\x)\cdot\rhoe^0(\x)\,d\x\\
\leq&\, \|\grad\u(t,\cdot)\|_{L^\infty}\int_{\Gamma_T\times\Omega} \big|\X(t,\x)-\gamma(t)\big|\,\eta_\eps(d\gamma,\x)\cdot\rhoe^0(\x)\,d\x\\
&+\int_\Omega|\u(t,\x)-\ue(t,\x)|\rhoe(t,\x)\,d\x\\
\leq &\, M B(t)+\sqrt{\etae(t)}.
\end{align*}
Together with $B(0)=0$, we apply Gr\"onwall inequality and obtain 
\begin{equation}\label{eq:W12}
W_1(\rhot(t,\cdot),\rhoe(t,\cdot))\leq B(t)\leq \int_0^t e^{M(t-s)}\sqrt{\etae(s)}\,ds\leq \sqrt{t}e^{Mt}\left(\int_0^t\etae(s)\,ds\right)^{\frac12}.	
\end{equation}
Finally, we apply \eqref{eq:W11} and \eqref{eq:W12} to \eqref{eq:W1decomp}, yielding
\[W_1^2(\rho(t,\cdot),\rhoe(t,\cdot))\leq 2(1+t)e^{2Mt}\left(W_1^2(\rho^0,\rhoe^0)+\int_0^t\etae(s)\,ds\right).\]
This finishes the proof of \eqref{eq:W1bound}, with $C=C(T,M)=2(1+T)e^{2MT}$.
\end{proof}

The initial distance $W_1^2(\rho^0,\rhoe^0)$ is small due to the assumption \eqref{eq:initapp}. Indeed, we have
\begin{align*}
W_1(\rho^0,\rhoe^0)=&\,\sup_{[g]_{Lip}\leq1}\int_\Omega g(\x)(\rho^0(\x)-\rhoe^0(\x))\,d\x\\
=&\,\sup_{[g]_{Lip}\leq1}\int_{\Omega\times\R^d} g(\x)(f^0(\x,\v)-\fe^0(\x,\v))\,d\x d\v\leq W_1(f^0,\fe^0)<\eps.
\end{align*}

Adding \eqref{eq:etaestint2} and \eqref{eq:W1bound}, we arrive at the inequality
\[
\etae(t)+W_1^2(\rho(t,\cdot),\rhoe(t,\cdot))
    \leq C\left[\eps+\eps^{\frac{q}{2-q}}+\eps^2+\int_0^t\Big(\etae(s)+ W_1^2(\rho(s,\cdot),\rhoe(s,\cdot))\Big)\,ds\right].
\]
Applying Gr\"onwall inequality, we end up with
\begin{equation}\label{eq:rhouconv}
    \etae(t)+W_1^2(\rho(t,\cdot),\rhoe(t,\cdot)) \leq Ce^{Ct}(\eps^{\frac{q}{2-q}}+\eps+\eps^2)\xrightarrow{~~\eps\to0~~} 0,
\end{equation}
for any $t\in[0,T]$. Here, the power $\frac{q}{2-q}\in(0,\frac12]$ since $q\in(0,1]$.

\subsection{Proof of Theorem \ref{thm:main}}
Now we apply the estimate \eqref{eq:rhouconv} to obtain our main convergence result \eqref{eq:convergence}.

Let $g=g(\x,\v)$ be a test function such that $[g]_{Lip_{\x,\v}}\leq 1$. 
In the following calculation, we fix a time $t$ and suppress the $t$-dependence. Compute
\begin{align*}
&\int_{\Omega\times\R^d} g(\x,\v) (f(\x,\v)-\fe(\x,\v))\,d\x d\v=
\int_{\Omega\times\R^d} g(\x,\u(\x)) (f(\x,\v)-\fe(\x,\v))\,d\x d\v\\
&\quad +\int_{\Omega\times\R^d} (g(\x,\v)-g(\x,\u(\x))) (f(\x,\v)-\fe(\x,\v))\,d\x d\v=: K_1+K_2.
\end{align*}
We estimate term by term. For $K_1$, we have
\[
|K_1|=\,\left|\int_\Omega g(\x,\u(\x)) (\rho(\x)-\rhoe(\x))\,d\x\right|\leq (1+[\u]_{Lip})W_1(\rho,\rhoe).
\]
For $K_2$, note that
\[
\int_{\Omega\times\R^d} (g(\x,\v)-g(\x,\u(\x))) f(\x,\v)\,d\x d\v=\int_{\Omega\times\R^d} (g(\x,\u(\x))-g(\x,\u(\x))) \rho(\x)\,d\x=0.
\]
Therefore,
\begin{align*}
|K_2|=&\,\left|\int_{\Omega\times\R^d} (g(\x,\v)-g(\x,\u(\x))) \fe(\x,\v)\,d\x d\v\right|
\leq\int_{\Omega\times\R^d}|\v-\u(\x)|\fe(\x,\v)\,d\x d\v\\
\leq&\,\int_{\Omega\times\R^d}|\v-\ue(\x)|\fe(\x,\v)\,d\x d\v+\int_{\Omega}|\ue(\x)-\u(\x)|\rhoe(\x)\,d\x\leq\Dee^{\frac12}+\etae^{\frac12}.
\end{align*}
Combine the estimates above and take supreme over all test functions $g$. We obtain
\[
W_1(f,\fe)\leq (1+M) W_1(\rho,\rhoe) + \etae^{\frac12} +\Dee^{\frac12}.
\]

From \eqref{eq:energydecay}, we have the control
\[\int_0^T\Dee(t)\,dt\leq\eps\,\Ee^0\leq C\eps.\]
Together with \eqref{eq:rhouconv}, we deduce the bound
\begin{equation}\label{eq:rhouconvrate}
\int_0^T W_1^2(f(t),\fe(t))\,dt
\leq C(\eps^{\frac{q}{2-q}}+\eps+\eps^2)\xrightarrow{~~\eps\to0~~} 0,	
\end{equation}
where the constant $C$ depends on $p, T, M$ and initial data. Therefore, we conclude with the convergence
\[W_1(f(t),\fe(t))\to0,\quad \text{in}~L^2(0,T).\]
This leads to the convergence result \eqref{eq:convergence}. To see this, we consider a test function $g=g(t,\x,\v)$ in $C([0,T], Lip(\Omega\times\R^d))$. Thus computing,
\begin{align*}
 &\,\int_0^T\int_{\Omega\times\R^d}g(t,\x,\v)(f(t,\x,\v)-\fe(t,\x,\v))\,d\x d\v dt\\
 \leq&\,\int_0^t [g(t)]_{Lip_{\x,\v}} W_1(f(t),\fe(t))\,dt
 \leq \big\|[g(t)]_{Lip_{\x,\v}}\big\|_{L^2(0,T)}\big\|W_1(f(t),\fe(t))\big\|_{L^2(0,T)}\to0.
\end{align*}
Note that Lipschitz functions in a bounded domain are dense in the space of continuous functions.
From Proposition \ref{prop:MP}, we know that the measures $f(t)$ and $\fe(t)$ are compactly supported. Then we apply the density argument and obtain convergence for any test function $g\in C_c([0,T]\times\Omega\times\R^d)$.

\begin{remark}
 We would like to mention another type of relative entropy, referred to as the \emph{kinetic relative entropy}:
  \[\etae^K(t) = \frac12\int_{\Omega\times\R^d} |\v-\u(t,\x)|^2\fe(t,\x,\v)\,d\x d\v.\]
 In \cite{shvydkoy2021dynamics}, Shvydkoy applies the relative entropy method based on $\etae^K$ to derive the hydrodynamic limit for the kinetic flocking model with linear velocity alignment ($p=2$). This approach can be adapted to the nonlinear case ($p>2$), using a similar argument to handle the discrepancy term. 
\end{remark}

\section{Isothermal hydrodynamic limit}\label{sec:isothermal}
In this section, we discuss another type of hydrodynamic limit, by isothermal ansatz \eqref{eq:isothermal}. The main point we would like to make is that the alignment force $\A[\rho,\u]$ that appears in the limiting system 
\[\pa_t(\rho\u)+\grad_\x\cdot(\rho\u\otimes\u)+\grad_\x p=\rho\A[\rho,\u]\]
does not necessarily share the same mapping $\Phi$ as the kinetic equation \eqref{eq:feFP}. 

To illustrate this point, we perform the following formal calculation.

Multiplying \eqref{eq:feFP} by $\v$ and integrating in $\v$, we obtain the momentum equation \eqref{eq:momentumeps}, with the right-hand side 
\[
\mathcal{F}(t,\x):=\int_{\R^d}\F(\fe)\fe\,d\v=\int_{\Omega\times\R^{2d}}\phi(\x-\y)\Phi(\w-\v)f(t,\x,\v)f(t,\y,\w)\,d\y d\v d\w.
\]
We suppress the $t$-dependency in the following calculation for simplicity. Applying isothermal ansatz \eqref{eq:isothermal}, we get
\begin{align*}
\mathcal{F}(\x)=&\,(2\pi)^{-d}\int_{\Omega\times\R^{2d}}\phi(\x-\y)\Phi(\w-\v)e^{-\frac{|\v-\u(\x)|^2+|\w-\u(\y)|^2}{2}}\rho(\x)\rho(\y)\,d\y d\v d\w\\
=&\,\frac{\rho(\x)}{(2\pi)^d}\int_{\Omega\times\R^{2d}}\phi(\x-\y)\Phi\big((\w-\v)+(\u(\y)-\u(\x))\big)e^{-\frac{|\v|^2+|\w|^2}{2}}\rho(\y)\,d\y d\v d\w.
\end{align*}
Substituting $\mathbf{a}=\w-\v$ and $\mathbf{b}=\w+\v$ yields
\begin{align*}
\mathcal{F}(\x)=&\,\frac{\rho(\x)}{(4\pi)^d}\int_{\Omega\times\R^{2d}}\phi(\x-\y)\Phi\big(\mathbf{a}+(\u(\y)-\u(\x))\big)e^{-\frac{|\mathbf{a}|^2+|\mathbf{b}|^2}{4}}\rho(\y)\,d\y d\mathbf{a} d\mathbf{b}\\
=&\,\frac{\rho(\x)}{(4\pi)^{d/2}}\int_{\Omega\times\R^d}\phi(\x-\y)\Phi\big(\mathbf{a}+(\u(\y)-\u(\x))\big)e^{-\frac{|\mathbf{a}|^2}{4}}\rho(\y)\,d\y d\mathbf{a}\\
=&\,\rho(\x)\int_\Omega\phi(\x-\y)\Psi(\u(\y)-\u(\x))\rho(\y)\,d\y,
\end{align*}
where the mapping $\Psi$ is defined as
\begin{equation}\label{eq:Psi}
 \Psi(\z) = \frac{1}{(4\pi)^{d/2}}\int_{\R^d}\Phi(\mathbf{a}+\z)e^{-\frac{|\mathbf{a}|^2}{4}}\,d\mathbf{a}=(\Phi\ast\mathcal{M}) (\z),
\end{equation}
where $\ast$ stands for convolution, and the function $\mathcal{M}$ is defined as 
\[\mathcal{M}(\z)=\frac{1}{(4\pi)^{d/2}}e^{-\frac{|\z|^2}{4}}.\]

One easy observation is that $\Psi$ is an odd mapping like $\Phi$. When $\Phi(\z)=\z$ is linear, we have
\[\Psi(\z)=\frac{1}{(4\pi)^{d/2}}\int_{\R^d}(\mathbf{a}+\z)e^{-\frac{|\mathbf{a}|^2}{4}}\,d\mathbf{a}=0+1\cdot\z=\z.\]
However, when $\Phi$ is nonlinear, the mapping $\Psi$ is \emph{not} the same as $\Phi$. To illustrate this point, we perform the following explicit calculate on $\Psi$ in the case: 
\[\Phi(z) = |z|^{p-2}z = z^{2k-1},\]
where $p=2k$ is an even integer, and $d=1$. Compute
\begin{align*}
 \Psi(z) & = \int_\R (z+a)^{2k-1}\mathcal{M}(a)\,da=\sum_{j=0}^{2k-1}{2k-1\choose j}z^{2k-1-j}\int_\R a^j\mathcal{M}(a)\,da\\
 & = \sum_{j=0}^{k-1}{2k-1\choose 2j}\frac{2^{2j}\Gamma(j+\frac12)}{\Gamma(\frac12)}\,z^{2k-1-2j},
\end{align*}
where we have used the following identity on moments of normal distribution:
\[
\int_\R z^k\mathcal{M}(z)\,dz = \begin{cases}
 \frac{2^k\Gamma(\frac{k+1}{2})}{\Gamma(\frac{1}{2})}& k\text{ is even,}\\
 0& k\text{ is odd.} 	
\end{cases}
\]
For example, when $\Phi(z) = z^3$ ($p=4$, or $k=2$), we have
\[\Psi(z) = z^3 + 6z.\] 
Unlike $\Phi$, the nonlinearity $\Psi$ behaves like $\mathcal{O}(z)$ instead of $\mathcal{O}(z^3)$ near $z=0$. Therefore, the hydrodynamic limit of \eqref{eq:feFP} does \emph{not} have the same $p$-alignment force.

The rigorous justification of the isothermal hydrodynamic limit will be left for future investigation.

\bibliographystyle{plain}
\bibliography{bib_hydro}

\end{document}